\documentclass[10pt,reqno]{amsart}

\usepackage{amsmath,amssymb,amsfonts,amsthm, amscd, enumerate, mathrsfs}
\usepackage[colorlinks=true, linkcolor=blue, citecolor=cyan, urlcolor=cyan, pagebackref=true]{hyperref}

\usepackage{tikz}
\usetikzlibrary{cd}

\usepackage{color}

\theoremstyle{plain}
\newtheorem{theorem}{Theorem}
\newtheorem{lemma}[theorem]{Lemma}
\newtheorem{corollary}[theorem]{Corollary}
\newtheorem{proposition}[theorem]{Proposition}

\newtheorem{prevtheorem}{Theorem}

\theoremstyle{definition}
\newtheorem*{definition*}{Definition}
\newtheorem*{example*}{Example}

\newcommand{\nwc}{\newcommand}

\nwc{\F}{\mathbb F}
\nwc{\N}{\mathbb N}
\nwc{\Z}{\mathbb Z}
\renewcommand{\P}{\mathbb P}

\nwc{\PGU}{\operatorname{PGU}}
\nwc{\Aut}{\operatorname{Aut}}
\nwc{\lub}{\operatorname{lub}}
\nwc{\ord}{\operatorname{ord}}
\renewcommand{\div}{\operatorname{div}}

\nwc{\mc}[1]{\mathcal{#1}}
\nwc{\scr}[1]{\mathscr{#1}}

\nwc{\la}{\langle}
\nwc{\ra}{\rangle}

\definecolor{mygreen}{rgb}{0.0, .7, 0.2}
\nwc{\red}[1]{\textcolor{red}{#1}}
\nwc{\blue}[1]{\textcolor{blue}{#1}}
\nwc{\green}[1]{\textcolor{mygreen}{#1}}


\begin{document}

\thispagestyle{empty}

\title[Triples of rational points on the Hermitian curve]{Triples of rational points on the Hermitian curve and their Weierstrass semigroups}
\date{}

\author{Gretchen L. Matthews}
\author{Dane Skabelund}
\author{Michael Wills}

\address{Department of Mathematics, Virginia Tech, Blacksburg, VA 24061 USA}
\email{gmatthews@vt.edu, dskabelund@vt.edu, mwills24@vt.edu}
\thanks{The first and third author were partially supported by NSF DMS-1855136.}

\maketitle

\begin{abstract}
In this paper, we study configurations of three rational points on the Hermitian curve over $\F_{q^2}$ and classify them according to their Weierstrass semigroups. For $q>3$, we show that the number of distinct semigroups of this form is equal to the number of positive divisors of $q+1$ and give an explicit description of the Weierstrass semigroup for each triple of points studied. To do so, we make use of two-point discrepancies and derive a criterion which applies to arbitrary curves over a finite field. 
\end{abstract}

\maketitle

\section{Introduction} 

The Weierstrass gap sequence of a point on an algebraic curve is a classically studied object which was generalized to the gap set of a pair of points by Arbarello, Cornalba, Griffiths, and Harris in 1985 \cite{ACGH} and that of an $n$-tuple of points for $n \geq 2$ in a series of works, including those by Ballico and Kim \cite{BK}, Iishi \cite{Ishii}, and Carvalho and Torres \cite{CarvalhoTorres}.  The gap sequence $G(P_1)$ at a single point $P_1$ has many properties which do not carry over to the gap set $G(P_1, \dots, P_n)$ of $n$ distinct points $P_1, \dots, P_n$.  For example, while the cardinality of $G(P_1)$ is equal to the genus of the curve for any point $P_1$, the size of $G(P_1, \dots, P_n)$ may depend on the choice of points $P_1, \dots, P_n$ for $n \geq 2$; see,  for instance, \cite{Carvalho_V}.

The complement of the gap set, the Weierstrass semigroup, of an $n$-tuple of points on a curve over finite field can be used to construct and decode algebraic geometry codes \cite{CF1, CF, CarvalhoTorres, GKL, KP, Korchmaros, PT, TT}, determine dimensions of associated Riemann-Roch spaces, and bound the number of rational points on a curve \cite{Geil}, among other applications. In this setting, Hasse-Weil maximal curves are particularly interesting as they yield long codes. It is for these reasons that we consider Weierstrass semigroups of rational points on the Hermitian curve over the field $\F_{q^2}$. 

The Weierstrass semigroup of a pair of rational points on the Hermitian curve was determined in \cite{pairs} and found to be independent of the choice of points. This is no longer the case for Weierstrass semigroups of triples of rational points on this curve, the subject of this paper. As we will see, triples of rational points on this curve admit a neat classification by their Weierstrass semigroups. To explain further, we now give explicit definitions of the objects involved. We also direct the reader to the end of this introduction, where we have collected an outline of our notational conventions.

Given an absolutely irreducible, smooth, projective algebraic curve $\mc X$ over a finite field $\F$ and a positive integer $n<|\F|$, the Weierstrass semigroup of the $n$-tuple $(P_1, \dots, P_n)$ of  distinct points of $\mc X$ is defined as
\[
  H(P_1, \dots, P_n)
= \left\{ {\alpha} \in \N^n : \sum_{i=1}^{n} \alpha_i P_i = (f)_{\infty} \text{ for some } f \in \F(\mc X) \right\} ,
\]
where $\N$ denotes the set of nonnegative integers.

The Riemann-Roch theorem assures that the complement 
\[
  G(P_1, \ldots, P_n) = \N^n \setminus H(P_1, \ldots, P_n)
\]
of the Weierstrass semigroup is finite. This complement is called the set of Weirstrass gaps, or just \emph{gaps}, of $(P_1, \dots, P_n)$.
Each gap $\alpha = (\alpha_1,\ldots,\alpha_n)$ corresponds to an effective divisor $D_\alpha = \alpha_1 P_1 + \cdots + \alpha_n P_n$ which has a base point at at least one of the points $P_j$, i.e., which satisfies 
\begin{equation} \label{gap}
  L(D_\alpha) = L(D_\alpha - P_j)
\end{equation}
for some $j$. 
A gap $\alpha$ is called a pure gap if (\ref{gap}) is satisfied for all $j \in \{ 1, \dots, n \}$, i.e., the correpsonding divisor $D_\alpha$ has a base point at each $P_j$.
We denote the set of pure gaps at $(P_1,\ldots,P_n)$ by $G_0(P_1,\ldots,P_n)$. Pure gaps provide the machinery to generalize the notion of consecutive gaps in \cite{GKL} to a multipoint setting, providing a much better bound on the error-correcting capabilty of associated algebraic geometry codes \cite{CarvalhoTorres}. 
In the same paper, it is shown that $\alpha$ is a pure gap for $(P_1,\ldots,P_n)$ if and only if 
\[
  L(D_\alpha) = L(D_\alpha -  P_1 - \cdots - P_n) .
\]

The Weierstrass semigroups have been determined for various collections of points on particular families of curves of interest in coding theory \cite{BMZ, BQZ, CB, Cas_Tiz, TizziottiCastellanos}.

For any automorphism $\sigma$ of $\mc X$, we note that
\begin{equation} \label{aut}
  H(P_1, \ldots, P_n)=H(\sigma(P_1), \ldots, \sigma(P_n)) .
\end{equation}
Thus, if $\Aut(\mc X)$ acts doubly transitively on the set of $\F$-rational points, there are at most $\#\mc X(\F) - 2$ distinct Weierstrass semigroups of triples of rational points of $\mc X$; indeed, if we fix two distinct rational points on $\mc X$, say $P$ and $Q$, then for any distinct triple $(P_1,P_2,P_3)$ of rational points, \eqref{aut} implies that
\[
H(P_1,P_2,P_3)=H(P,Q,R)
\]
for some point $R$ in $\mc X(\F) \setminus \{P,Q\}$.
For the Hermitian curve, we will use the geometry of the curve, along with two-point discrepancies, a concept introduced by Duursma and Park in 2012 \cite{Duursma_Park}, to considerably refine this bound. In particular, as the main result of the paper we determine the following.

\begin{theorem}\label{main_theorem}
For $q > 3$, the number of distinct Weierstrass semigroups of triples of rational points on the Hermitian curve over $\F_{q^2}$ is equal to the number of positive divisors of $q+1$.
For $q=2$, there is a single three-point semigroup, and for $q=3$, there are two.
\end{theorem}

In addition, we determine the cardinality of the associated sets of gaps and pure gaps, and give an explicit description of the Weierstrass semigroups themselves, by describing a subset which generates the semigroup with respect to the operation of taking coordinate-wise maximums.
Understanding these semigroups many be useful for studying three-point algebraic geometry codes on the Hermitian curve, complementing the approach in \cite{BR}.

The remainder of the paper is organized as follows. This section concludes with a summary of notation to be used throughout the paper. Section \ref{sec:triangles} describes the orbits of three-point configurations on the Hermitian curve under the action of the automorphism group and introduces an invariant which determines the associated Weierstrass semigroups. Section \ref{section:discrepancies} introduces two-point discrepancies and derives a criterion for identifying Weierstrass gaps. The results in this section apply to any smooth projective curve; in particular, Proposition \ref{basepoint} and Corollary \ref{corollary:basepoint_G} may be of independent interest. In Section \ref{sec:sigma}, we apply this approach to triples of rational points on the Hermitian curve to see that triangles of the same type have the same Weierstrass semigroup. Counting arguments are applied in Section \ref{section:counting} to determine the total numbers of gaps of a triples of rational points, completing the proof of the classification of triples of points according to their semigroups. The semigroups themselves are described in Section \ref{section:semigps}.

{\bf{Notation.}} 
We denote the nonnegative integers by $\N$ and the finite field of $q$ elements by $\F_q$.
The multiplicative group of  nonzero elements of a field $\F$ will be written as $\F^{\times}$.
We write $\ord(\zeta)$ to denote the multiplicative order of $\zeta \in \F_q^\times$.

Given a curve $\mc X$ over a field $\F$, we denote the set of $\F$-rational points by $\mc X(\F)$ and the function field of $\mc X$ by $\F(\mc X)$. 
For $f \in \F(X)^\times$, we write $\div(f)$ and $(f)_\infty$ for the divisor and pole divisor of $f$, respectively.
For $f \in \F(\mc X)$ and $P \in \mc X$, the valuation of $f$ at $P$ is denoted by $v_P(f)$.

For a divisor $A$ on $\mc X$, we write ${L}(A) = \left\{ f \in \F(\mc X)^\times : \operatorname{div}(f) \geq - A \right\} \cup \{ 0 \}$ for the Riemann-Roch space of the divisor $A$.
Given two divisors $A$ and $B$ on  $\mc X$, we write $A \sim B$ to mean that they are linearly equivalent, that is, that there is $f \in \F(\mc X)$ with $\operatorname{div}(f) = A-B$.

\section{Hermitian triangles}\label{sec:triangles}

For a prime power $q$, the Hermitian curve is the unique curve $\mc H = \mc H_q$ over $\F_{q^2}$, up to birational isomorphism, which has genus $g = q(q-1)/2$ and $q^3+1$ rational points \cite{RS}. These $q^3+1$ points admit a doubly transitive action of the automorphism group of $\mc H$, which is isomorphic to $\PGU(3,q^2)$ \cite{St}.
The curve $\mc H$ is isomorphic to the Fermat curve of degree $q+1$.
However, in this work, we will use the smooth plane model for $\mc H$ defined by the vanishing of
\begin{equation} \label{defining_eq}
  F(X,Y,Z) = X^{q+1} - Y^q Z - Y Z^q .
\end{equation}
This model has the benefit of admitting a natural choice of two rational points $P_\infty = (0:1:0)$ and $P_{00} = (0:0:1)$, which will prove useful for dealing with the automorphism group explicitly. We denote the $q^3$ affine points of $\mc H(\F_{q^2}) \setminus \{P_\infty\}$ by $P_{\alpha\beta} = (\alpha:\beta:1)$.

The rational functions $x = X/Z$ and $y = Y/Z$ satisfy $x^{q+1} = y^q + y$ and have divisors
\[
\operatorname{div}(x) = - qP_\infty + \sum_{\beta^q + \beta = 0} P_{0\beta}
\qquad\text{and}\qquad
\operatorname{div}(y) = m(P_{00} - P_\infty) ,
\]
where $m = q+1$.
Since the automorphism group acts doubly transitive on $\mc H(\F_{q^2})$, composing the function $y$ with an appropriate automorphism yields a function $y_{AB}$ with divisor
\[
\operatorname{div}(y_{AB}) = m(B-A) 
\]
for any $A,B \in \mc H(\F_{q^2})$. 
Thus, the $m$th multiples of any two $\F_{q^2}$-rational points of $\mc H$ are linearly equivalent.
If $A = P_\infty$ and $B = P_{\alpha\alpha}$, we denote the function $y_{AB}$ by $y_\alpha$, and note that
\begin{equation}\label{yalpha}
y_\alpha
= y - \alpha- \alpha^q (x-\alpha)
= y - \alpha^q(x-1) .
\end{equation}

We are interested in studying Weierstrass semigroups $H(P,Q,R)$ associated to triples $(P,Q,R)$ of distinct $\F_{q^2}$-rational points of the Hermitian curve. Note that the  $\F_{q^2}$-rational points of $\mathcal H$ are precisely its Weierstrass points \cite{Garcia_Viana}.
We will see in Corollary \ref{corollary:S3} that (perhaps unexpectedly) the semigroup $H(P,Q,R)$ of any three rational points of $\mc H$ is invariant under any permutation of the three points, so that $H(P,Q,R)$ depends only on the three-element set $T = \{P,Q,R\}$, which we call a Hermitian triangle.
In this paper, the term Hermitian triangle will only be used to refer sets of points defined over $\F_{q^2}$.
In the degenerate case that $P$, $Q$, and $R$ lie on a line in $\P^2$, we call $T$ a collinear triangle. 

For any Hermitian triangle $T = \{P,Q,R\}$ and automorphism $\sigma$ of $\mc H$, Corollary \ref{corollary:S3} and \eqref{auts} together imply that $T$ and $\sigma(T)$ define the same semigroup.
It is therefore natural to consider triangles up to the action of $\Aut(\mc H)$.
We say that two triangles $T$ and $T'$ are in the same automorphism class if there is $\sigma \in \Aut(\mc H)$ such that $\sigma(T) = T'$.

The automorphism group of the Hermitian curve is well understood; see, for instance, \cite{St}. For our purposes, we make use of 
the following explicit representation of $\Aut(\mc H)$ which was shared with the first author by John Little, along with the results in Lemmas \ref{3cycle} and \ref{orbits} \cite{Little}.

\begin{lemma}\label{auts}
Let $(\alpha:\beta:\gamma)$ and $(\lambda:\mu:\nu)$ be any two distinct points of $\mc H(\F_{q^2})$, and let $\epsilon \in \F_{q^2}^\times$.
Then there is an automorphism of $\mc H$ induced by the linear mapping on $\P^2$ defined by left multiplication by the matrix
\[
  M = \begin{bmatrix}
\epsilon(\gamma \mu - \beta \nu)^q & \epsilon^{q+1} \xi \lambda & \alpha \\
\epsilon(\alpha \mu - \beta \lambda)^q & \epsilon^{q+1} \xi \mu & \beta \\
\epsilon(\gamma\lambda - \alpha \nu)^q & \epsilon^{q+1} \xi \nu & \gamma 
\end{bmatrix} ,
\]
where $\xi = -\lambda^q \alpha + \mu^q \gamma + \nu^q \beta$.
Moreover, every element of $\Aut(\mc H)$ can be written in this form for some choice of two points and $\epsilon$.
\end{lemma}

\begin{proof}

Let $(x:y:z)$ be a point of $\P^2$ with image $(x':y':z')$ under $M$.
Then one may check directly that 
\[
  F(x',y',z') = \epsilon^{q+1} F(A,B,C) F(x,y,z),
\]
where $A = \gamma \mu - \beta \nu$, $B = \alpha \mu - \beta \lambda$ and $C = \gamma \lambda - \alpha \nu$, and $F$ is the defining polynomial of $\mc H$ given in \eqref{defining_eq}.
Thus, $M$ leaves $\mc H$ invariant.

We now show that the matrix $M$ is invertible, and hence defines an automorphism of $\mc H$.
Since
\[
  \det M 
= \epsilon^{q+2} \xi F(A,B,C) 
= \epsilon^{q+2} \xi^{q+2}
\]
and $\epsilon \neq 0$, we want to show that $\xi \neq 0$.
To do so, consider the sesquilinear form $\Psi \colon  \F_{q^2}^3 \times  \F_{q^2}^3 \to \F_{q^2}$ defined by
\[
  \Psi(v,w) = v_1^qw_1 - v_2^q w_3 - v_3^q w_2,
\]
and note that $F(v) = \Psi(v,v)$.
If $v,w \in  \F_{q^2}^3$ satisfy $\Psi(v,v) = 0$, $\Psi(w,w) = 0$, and $\Psi(v,w) = 0$, then for any $a,b \in \F_{q^2}$, we have
\begin{align*}
  F(av+bw)
&= \Psi(av,av) + \Psi(av,bw) + \Psi(bw,av) + \Psi(bw,bw) \\
&= a^{q+1} \Psi(v,v) + a^q b \Psi(v,w) + ab^q \Psi(v,w)^q + b^{q+1} \Psi(w,w) = 0 .
\end{align*}

In particular, this means that if the two linearly independent vectors $v = (\lambda,\mu,\nu)$ and $w = (\alpha,\beta,\gamma)$ satisfy
\[
  \xi = - \Psi(v,w) = 0,
\]
then $\mc H$ contains all $q^2+1$ of the $\F_{q^2}$-rational point on the line through $(\alpha:\beta:\gamma)$ and $(\lambda:\mu:\nu)$ in $\P^2$. However, this is not the case, as Bezout's Theorem assures that any line intersects $\mc H$ in at most $q+1$ points. 

To prove the final claim, we note that there are $(q^3+1) q^3 (q^2-1)$ choices for the two points $(\alpha:\beta:\gamma)$ and $(\lambda:\mu:\nu)$ and the unit $\epsilon$, each yielding a distinct automorphism of $\mc H$. Since this is equal to the order of $\Aut(\mc H)$, we have accounted for all elements of the group. \qedhere
\end{proof}

Note that the matrix $M$ maps the points $P_{00}$ and $P_\infty$ to $(\alpha:\beta:\gamma)$ and $(\lambda:\mu:\nu)$, respectively, witnessing the double-transitivity of $\Aut(\mc H)$ on $\mc H(\F_{q^2})$. Moreover, the stabilizer of the two points $P_{00}$ and $P_\infty$ in $\Aut(\mc H)$ consists of automorphisms $\phi_\epsilon(X:Y:Z) = (\epsilon X: \epsilon^{q+1}Y :Z)$ for $\epsilon \in  \F_{q^2}^\times$.

Since all automorphisms of $\mc H$ are induced by linear mappings on $\P^2$, the image of a collinear triangle under an automorphism is again collinear.

By the double transitivity of $\Aut(\mc H)$ on $\mc H(\F_{q^2})$, each automorphism class of Hermitian triangles contains a triangle of the form $T = \{ P_\infty, P_{00}, P_{\alpha\beta}\}$.
Such a triangle satisfies $\beta\neq 0$, as $P_{00}$ is the only point of $\mc H$ on the line $y=0$.
If $\alpha \neq 0$, then acting on $T$ by $\phi_\epsilon$ with $\epsilon = (\beta \alpha^{-1})^{-q}$ produces a triangle with $\alpha = \beta$.

\begin{definition*}
A Hermitian triangle $T$ is in \emph{standard form} if $T = \{ P_\infty, P_{00}, P_{\alpha\beta}\}$, and either $\alpha = 0$ or $\alpha = \beta$.
\end{definition*}

For $T$ in standard form, the cases $\alpha = 0$ and $\alpha = \beta$ are mutually exclusive. If $\alpha = 0$, then $T$ is collinear, and if $\alpha = \beta$, then $T$ is noncollinear.

While the following fact is well-known, we include it and a short proof, since it is important for this study. 

\begin{lemma}\label{3cycle}
Given any Hermitian triangle $T = \{P,Q,R\}$, there is an automorphism of $\mc H$ which acts as a 3-cycle on $T$.
\end{lemma}

\begin{proof}
We may assume that $(P,Q,R) = (P_\infty, P_{00}, P_{\alpha\beta})$ with $\beta \neq 0$.
Then applying Lemma \ref{auts} with the points $P_\infty$ and $P_{\alpha\beta}$ and with $\epsilon = \beta^{-q}$ yields an automorphism $\sigma$ of $\mc H$ satisfying $\sigma(P_{00}) = P_\infty$, $\sigma(P_\infty) = P_{\alpha\beta}$, and $\sigma(P_{\alpha\beta}) = P_{00}$. \qedhere
\end{proof}

\begin{lemma}\label{orbits}
There are $1 + \lceil q/2 \rceil$ automorphism classes of $\F_{q^2}$-rational triangles on the Hermitian curve.
These classes are in correspondence with Galois-conjugacy classes of roots $\alpha$ of the polynomial $t^{q+1}- t^{q} - t$ over $\F_q$. In particular, the class corresponding to the root $\alpha = 0$ consists of all collinear Hermitian triangles, and the class corresponding to the conjugates of a nonzero root $\alpha$ is represented by the noncollinear triangle $\{P_\infty, P_{00}, P_{\alpha\alpha} \}$.

\end{lemma}

\begin{proof}

Since each automorphism class is represented by a triangle in standard form $T = \{ P_\infty, P_{00}, P_{\alpha\beta}\}$, it suffices to determine when two such triangles are in the same class.

If $\alpha = 0$, then $\beta^q + \beta = 0$ and $\beta \neq 0$, so that $\beta^{q-1}= -1$.
We claim that all such triangles are in the same automorphism class.
As $\epsilon$ runs over $\F_{q^2}^\times$, the powers of $\epsilon^{q+1}$ run over $\F_q^\times$, and $\beta \epsilon^{q+1}$ runs over all $(q-1)$st roots of $-1$.
Therefore, each of the $q-1$ points of $\mc H$ of the form $P_{0\beta'}$ with $\beta'\neq 0$ is realized as $\phi_\epsilon(P_{0\beta})$ for some $\epsilon$.

Triangles with $\alpha = \beta \neq 0$ are noncollinear; hence, they do not share an automorphism class with those with $\alpha = 0$.
We claim that triangles $T = \{ P_\infty, P_{00}, P_{\alpha\alpha}\}$ and $T' = \{ P_\infty, P_{00}, P_{\alpha'\alpha'}\}$ are in the same classs if and only if $\alpha' \in \{ \alpha, \alpha^q\}$.

Suppose that $\sigma \in \Aut(\mc H)$ satisfies $\sigma(T') = T$.
By Lemma \ref{3cycle}, we may assume that $\sigma$ fixes $P_\infty$.
Then $\sigma$ sends $P_{00}$ either to $P_{00}$ or $P_{\alpha\alpha}$.
If $\sigma(P_{00}) = P_{00}$, then $\sigma = \phi_\epsilon$ for some $\epsilon \in \F_{q^2}^\times$.
But then $\alpha = \epsilon \alpha' = \epsilon^{q+1} \alpha'$, so that $\epsilon^q = 1$.
Since $\epsilon$ has order coprime to $q$, this means that $\epsilon = 1$, so that $\alpha' = \alpha$.
On the other hand, if $\sigma(P_{00}) = P_{\alpha\alpha}$, then Lemma \ref{auts} implies that $\sigma$ is of the form
\[
\sigma(X:Y:Z) = (\epsilon X + \alpha Z : \epsilon \alpha^q X + \epsilon^{q+1} Y + \alpha Z : Z)
\]
for some $\epsilon \in \F_{q^2}^\times$.
Then from $\sigma(P_{\alpha'\alpha'}) = P_{00}$, we have
\[
\epsilon\alpha' + \alpha = 0
\qquad\text{and}\qquad
\epsilon\alpha^q \alpha' + \epsilon^{q+1}\alpha' + \alpha = 0 .
\]
Solving these equations over $\F_{q^2}$ yields $\alpha' = \alpha^q$, as desired. \qedhere

\end{proof}

Although Hermitian triangles in the same automorphism class have the same Weierstrass semigroup, there are generally strictly fewer distinct three-point semigroups than automorphism classes of Hermitian triangles \cite{JM}. 
We now define an invariant which will allow us to distinguish triangles with different semigroups.

By Lemma \ref{orbits}, the automorphism class of any Hermitian triangle $T$ is represented by a triangle in standard form $\{ P_\infty, P_{00}, P_{\alpha\beta}\}$, with $\alpha$ a root of the polynomial
\[
t^{q+1} - t^q - t = (t-1)^{q+1} - (-1)^{q+1} .
\]
Such $\alpha$ are precisely those of the form $\alpha = 1 - \zeta$ with $\zeta^{q+1} = 1$. For any such $\alpha$, the multiplicative order of $\zeta = 1-\alpha$ is a divisor $d$ of $q+1$.

\begin{definition*}
Let $T$ be a Hermitian triangle.
Then $T$ is of \emph{type $d$} if there is $\sigma \in \Aut(\mc H)$ such that $\sigma(T) = \{ P_\infty, P_{00}, P_{\alpha\beta}\}$ is in standard form, and $\ord(1-\alpha) = d$.
\end{definition*}

By Lemma \ref{orbits},  two triangles $\{ P_\infty, P_{00}, P_{\alpha\beta}\}$ and $\{ P_\infty, P_{00}, P_{\alpha'\beta'}\}$ in standard form are in the same automorphism class only if $\alpha$ and $\alpha'$ are Galois conjugates, in which case $\ord(1-\alpha) = \ord(1-\alpha')$. Thus, the type of a triangle is well-defined.

Note that triangles of type $d=1$ are precisely those triangles which are collinear.
We will use the notation $T_d$ throughout to denote triangles of type $d$.
Since there is a triangle of type $d$ for every divisor $d$ of $q+1$, to prove Theorem \ref{main_theorem} for $q>3$ it will suffice to show that the type $d$ of a Hermitian triangle characterizes its Weierstrass semigroup.

\section{Discrepancies and basepoints}\label{section:discrepancies}

The main tool we will use to study three-point semigroups is the notion of two-point discrepancies introduced by Duursma and Park \cite[\S 2--3]{Duursma_Park}. Here we recall the definition of discrepancies and their relevant properties.
All results in this section apply to any absolutely irreducible, smooth, projective algebraic curve $\mc X$.

Fix two distinct points $P$ and $Q$ on $\mc X$, and let $K$ be a canonical divisor of $\mc X$.
\begin{definition*}
Let $\Delta(P,Q)$ be the set of all divisors $D$ on $\mc X$ such that $L(D) \neq L(D-P)$ and $L(D-Q) = L(D-P-Q)$.
A divisor in $\Delta(P,Q)$ is called a \emph{discrepancy} for $P$ and $Q$.
\end{definition*}

This definition is symmetric in $P$ and $Q$, as can be seen by examining the following diagram.
\[
\begin{tikzcd}
& L(D) \arrow[dl, dash, "\neq"'] \arrow[dr, dash] & \\
L(D-P) \arrow[dr, dash] & & L(D-Q) \arrow[dl, dash, "="', above] \\
& L(D-P-Q) &
\end{tikzcd}
\]

As it is often easier to demonstrate an inequality of two spaces than an equality, the following criterion for identifying discrepancies will prove useful.
\begin{lemma}\label{fg_lemma}
A divisor $D$ is in $\Delta(P,Q)$ if and only if
\[
 L(D-P) \neq L(D)
 \quad\text{and}\quad
 L(K-D+Q) \neq L(K-D+P+Q) .
\]
\end{lemma}
\begin{proof}
It follows directly from the Riemann-Roch theorem that the condition $L(D-Q) = L(D-P-Q)$ is equivalent to $L(K-D+Q) \neq L(K-D+P+Q)$.
\end{proof}

The following dimension formula is the main reason that we care about discrepancies, as it will allow us to compare the dimensions of various function spaces in order to identify Weierstrass gaps.

\begin{lemma}\cite[Theorem 3.4]{Duursma_Park}\label{dim_lemma}
For a given divisor $B$, 
\begin{equation}\label{dimeq_1}
\dim L(B + aP + bQ) 
= \#\{ B + iP + jQ \in \Delta(P,Q)  : i \leq a, j \leq b\} .
\end{equation}
\end{lemma}

Let $\Delta_B(P,Q) = \Delta(P,Q) \cap \{ B + iP + jQ \colon i,j \in \Z\}$ denote the set of discrepancies counted in \eqref{dimeq_1}.

\begin{lemma}\cite[Theorem 3.5]{Duursma_Park} \label{sigmaB}
For a given divisor $B$, there is a bijective function $\sigma_B\colon \Z \to \Z$ such that 
\[
  \Delta_B(P,Q) = \{  B + iP + \sigma_B(i) Q \colon i \in \Z \} .
\]
For $m$ such that $mP \sim mQ$, the function $i + \sigma_B(i)$ depends only on $i$ modulo $m$.
Moreover, $\sigma_B$ is determined by its image on a full set of representatives modulo $m$, and for $m$ minimal such that $mP \sim mQ$, the set $\Delta_B(P,Q)$ consists of $m$ distinct divisor classes.
\end{lemma}

We emphasize that the function $\sigma_B$ depends not only on $B$, but on the points $P$ and $Q$ which respect to which discrepancies are being considered. In situations where the two points are not clear from context, we write $\sigma_B = \sigma_{B,P,Q}$.

Let us take a moment to reformulate the dimension formula from Lemma \ref{dim_lemma} in terms of the function $\sigma_B$.
For $0 \leq i \leq m-1$, the divisors in the divisor class of $B + iP + \sigma_B(i)Q$ which contribute to the count in \eqref{dimeq_1} correspond to integer solutions $k$ of the system of linear inequalities
\begin{align*}
i - km &\leq a \\
\sigma_B(i) + km &\leq b.
\end{align*}
Counting these solutions and summing over each divisor class yields the closed-form formula
\begin{equation}\label{dimeq_2}
  \dim L(B + aP + bQ) 
= \sum_{\substack{i \bmod m \\ i + \sigma_B(i) \leq a+b}} \left( \left\lfloor \frac{a-i}{m} \right\rfloor + \left\lfloor \frac{b-\sigma_B(i)}{m} \right\rfloor + 1 \right) .
\end{equation}

\begin{proposition}\label{basepoint}
A divisor $B + aP + bQ$ has a basepoint at $P$ if and only if $\sigma_B(a) > b$.
\end{proposition}

\begin{proof}
The divisor $D = B + aP + bQ$ has a basepoint at $P$ if and only if $\dim(D) = \dim (D-P)$.
By \eqref{dimeq_2}, this may be written as
\begin{multline}\label{eqn:sums}
\sum_{\substack{i \bmod m \\ i + \sigma_B(i) \leq a+b}} \left( \left\lfloor \frac{a-i}{m} \right\rfloor + \left\lfloor \frac{b-\sigma_B(i)}{m} \right\rfloor + 1 \right) \\
= \sum_{\substack{i \bmod m \\ i + \sigma_B(i) \leq (a-1)+b}} \left( \left\lfloor \frac{(a-1)-i}{m} \right\rfloor + \left\lfloor \frac{b-\sigma_B(i)}{m} \right\rfloor + 1 \right) .
\end{multline}

Note that the slightly stronger condition on $i$ on the right hand side excludes those indices $i$ with $i + \sigma_B(i) = a+b$.
Suppose that $i$ satisfies $i  + \sigma_B(i) = a+b$.
Then the corresponding term on the left hand side of \eqref{eqn:sums} is 
\[
\left\lfloor \frac{a-i}{m} \right\rfloor + \left\lfloor \frac{i-a}{m} \right\rfloor + 1
= \begin{cases}
1, & i \equiv a \bmod m \\
0, & i \not \equiv a \bmod m .
\end{cases}
\]
If $i \equiv a\bmod m$, then \eqref{eqn:sums} cannot hold, because each of the remaining terms on the left hand side of \eqref{eqn:sums} is at least as large as the correpsponding term on the right hand side.
On the other hand, if $i \not\equiv a \bmod m$, the extra term on the left hand side does not effect the sum.
Upon excluding these terms, \eqref{eqn:sums} becomes
\[
\sum_{\substack{i \bmod m \\ i + \sigma_B(i) < a+b}} \left\lfloor \frac{a-i}{m} \right\rfloor 
= \sum_{\substack{i \bmod m \\ i + \sigma_B(i) < a+b}}  \left\lfloor \frac{(a-1)-i}{m} \right\rfloor .
\]

There is only a single index $i \bmod m$ for which the terms on each side differ, namely $i \equiv a \bmod m$.
Thus, equality holds exactly when this term does not appear, which is when $a + \sigma_B(a) > a + b$, or $\sigma_B(a) > b$. 
\end{proof}

Noting that $(a,b,c) \in \N^3$ is in the Weierstrass gap set $G(P,Q,R)$ if and only if the divisor $aP+bQ+cR$ has a basepoint at one of $P,Q,R$ yields the following.

\begin{corollary}\label{corollary:basepoint_G} The following are equivalent: 
\begin{enumerate}
\item The triple $(a,b,c) \in \N^3$ is in the Weierstrass gap set $G(P,Q,R)$. 
\item The divisor $aP + bQ+cR$ has a basepoint at one of $P$, $Q$, or $R$. 
\item Given points $P$, $Q$, and $R$, and $(a,b,c) \in \N^3$, $\sigma_{cR,P,Q}(a) > b$, $\sigma_{aP,Q,R}(b) > c$, or $\sigma_{bQ,R,P}(c) > a$. 
\end{enumerate}
\end{corollary}

\section{Discrepancies for Hermitian triangles}\label{sec:sigma}

In this section, we study discrepancies on the Hermitian curve which are supported on a Hermitian triangle.

\begin{definition*}
Let $T = \{P, Q, R\}$ be a Hermitian triangle.
For any integers $i$ and $j$, let $\sigma_{ij} = \sigma_{ij}(P,Q,R)$ be the unique integer satisfying
\[
iP + jQ + \sigma_{ij} R \in \Delta(Q,R).
\]
\end{definition*}

Note that this definition depends, a priori, on the particular ordering of the points $P$, $Q$, and $R$. 
We will see in Theorem \ref{thm:sigma}, however, that it is independent of this ordering, and that the function $(i,j) \mapsto \sigma_{ij}(T)$ depends only on the type $d$ of the triangle $T$.

The $\sigma_{ij}$ are related to the function $\sigma_B$ from Lemma \ref{sigmaB} as follows.
Since we are dealing with discrepancies for the points $Q$ and $R$, these two points take the roles of $P$ and $Q$ in the discussion in Section \ref{section:discrepancies}.
Moreover, with $B = iP$, we have $\sigma_{B}(j) = \sigma_{ij}$ as defined above.

By Lemma \ref{sigmaB}, the $\sigma_{ij}$ are determined by their values on any complete set of representatives $(i,j)$ modulo $m$.
In particular, for any integers $s$ and $t$, we have
\begin{equation}\label{eqn:sigma_ijm}
  \sigma_{i + sm, j + tm} = \sigma_{ij} - sm - tm .
\end{equation}

The function $(i,j) \mapsto \sigma_{ij}$ admits a particularly simple description on the domain
\[
  \Lambda =  \{ (i,j) : 1 \leq j, i-j+1 \leq q+1\} ,
\]
which we describe in the following theorem.

\begin{theorem}\label{thm:sigma}
Let $T$ be a Hermitian triangle. 
Then the function $(i,j) \mapsto \sigma_{ij}(T)$ is well-defined, and depends only on the type of $T$.
In particular,
\begin{enumerate}[(a)]
\item If $T = T_1$ is collinear, then for all $(i,j) \in \Lambda$, we have
\begin{equation}\label{eqn:sigma}
\sigma_{ij}(T_1) = 2g - 2 - jq .
\end{equation}
\item If $T = T_d$ is of type $d$, then 
\begin{equation}\label{sigma_diff}
\sigma_{ij}(T_d) =
\begin{cases}
\sigma_{ij}(T_1) + 1, & i \not\equiv 0 \bmod d \text{ and } j \equiv i \bmod m, \\
\sigma_{ij}(T_1) - 1, & i \not\equiv 0 \bmod d \text{ and } j \equiv i + 1 \bmod m, \\
\sigma_{ij}(T_1),  & \text{otherwise.}
\end{cases}
\end{equation}
\end{enumerate}
\end{theorem}

Before proving Theorem \ref{thm:sigma}, we provide two lemmas along with the following discussion as preparation. 
We first resolve the question of well-definedness of $\sigma_{ij}(T)$, under the assumption that parts (a) and (b) of Theorem \ref{thm:sigma} hold for a particular ordering $(P,Q,R)$ of the triangle $T$.
By Lemma \ref{3cycle}, it will be enough to show that $\sigma_{ij}$ is invariant under a transposition of two of the points of $T$, e.g., that
\[
  \sigma_{ij}(P,Q,R) = \sigma_{ij}(P,R,Q) .
\]

Since $\Delta_{iP}(Q,R)$ is a union of divisor classes, it will suffice to show that
\begin{equation}\label{QR}
  j Q + \sigma_{ij} R \sim \sigma_{ij} Q + jR ,
\end{equation}
with $\sigma_{ij}$ as given by the theorem.

Suppose first that $T$ is collinear, so that $\sigma_{ij} = 2g - 2 - jq$ for $(i,j) \in \Lambda$.
Then since $2g-2 = (q-2)m$, we have $2g - 2 - jq \equiv j \bmod m$, and the two divisors in \eqref{QR} differ by a multiple of $m(Q-R) \sim 0$.

Now suppose that $T$ is of type $d$. If $i \equiv 0 \bmod d$ or $j$ is neither $i$ nor $i+1$ modulo $m$, then we have the same formula for $\sigma_{ij}$, and the same argument applies. Otherwise, we have
\begin{align*}
  iQ + \sigma_{ii}(T) R
&= iQ + (\sigma_{ii}(T_1)+1) R \\
&\sim \sigma_{ii}(T_1) Q + (i+1) R \\
&= (\sigma_{i,i+1}(T_1) - 1) Q + (i+1) R
= \sigma_{i,i+1}(T) Q + (i+1) R ,
\end{align*}
which deals with both the cases $j \equiv i \bmod m$ and $j\equiv i+1 \bmod m$.

Thus, to prove the theorem, we may assume without loss of generality that $T = \{P,Q,R\}$ is of standard form, i.e., that $P = P_\infty$, $Q = P_{00}$, and $R = P_{\alpha\beta}$ with either $\alpha = 0$ or $\alpha = \beta$. We now pause to present two lemmas that will be of use in the proof of Theorem \ref{thm:sigma}.

\begin{lemma}\label{shift}
Let $T = \{P,Q,R\}$ be a Hermitian triangle of type $d$. Then there is $w_d \in \F_{q^2}(\mc H)$ with divisor 
\[
  \operatorname{div}(w_d) = -qdP + dQ + dR + E, 
\]
where $E \geq 0$ has support disjoint from $T$.
\end{lemma}

\begin{proof}

If $d = q+1$, then let $S$ be any point of $\mc H(\F_{q^2}) \setminus T$.
Then the function $w_{q+1} = y_{PQ} y_{PR} y_{PS}^{q-2}$ has the desired divisor.

Now assume that $d < q+1$. With the conventions on $P,Q,R$ described above, we set
\[
w_d = (x-y)^d - (x-y-1)^d + (x-1)^d .
\]

If $d = 1$, then the function $w_d = x$ is as desired.
Assume that $1 < d < q+1$. The function $w_d$ is a polynomial in $x$ and $y$ of degree $d$, hence is regular away from $P = P_\infty$.
Since $d < q+1$, the pole orders $iq + j(q+1)$ at $P$ of the monomials $x^i y^j$ appearing in $w_d$ are distinct.
Moreover, the $y^d$ and $xy^{d-1}$ terms of $w_d$ are zero, so that
\[
v_{P}(w_d) 
= \min(v_{P}(x^d),  v_{P}(y^{d-1}) )
= \min( -dq, -dq + q+1-d)
= - dq .
\]

To check the valuations of $w_d$ at the points $Q$ and $R$, we note that 
\begin{align*}
  w_d 
&= (x-y)^d + \prod_{\zeta^d = 1} \Big( (x-1) - \zeta (x-y-1) \Big) \\
&= (x-y)^d + \prod_{\zeta^d = 1} \zeta(y - (1-\zeta^{-1})(x-1)) \\
&= (x-y)^d + \prod_{ \ord(1-\eta) \mid d} (1-\eta^q)^{-1} y_{\eta} .
\end{align*}

For each $\eta$ with $\ord(1-\eta)$ dividing $d$, the product on the right hand side is divisible by $y_\eta$, so it vanishes at $P_{\eta\eta}$ to order at least $q+1 > d$. Therefore, 
\[
  v_{P_{\eta\eta}}(w_d) = v_{P_{\eta\eta}}( (x-y)^d) = d .
\]
Since $Q = P_{00}$ and $R = P_{\alpha\alpha}$ are both of this form, $v_Q(w_d) = v_R(w_d) = d$. \qedhere

\end{proof}

To prove the rest of Theorem \ref{thm:sigma}, it will suffice by Lemma \ref{fg_lemma} to show for each $(i,j) \in \Lambda$ that 
\[
 L(D-R) \neq L(D)
 \quad\text{and}\quad
 L(K-D+Q) \neq L(K-D+Q+R) ,
\]
where $D = iP + jQ + \sigma_{ij} R$ and $K = (2g-2)R$.

That said, we may actually choose a single one of these conditions to verify for all $(i,j) \in \Lambda$.
This is because the spaces involved exhibit a duality with respect to the involution $I \colon \Lambda \to \Lambda$ defined by
\[
  I (i,j) = (2m-i, m+1-j) .
\]

The map $I$ interchanges the two subsets
\begin{align*}
  \Lambda_d^+ &= \{ (i,j) \in \Lambda \colon i \not\equiv 0 \bmod d \text{ and } j \equiv i \bmod m\}, \\
  \Lambda_d^- &= \{ (i,j) \in \Lambda \colon i \not\equiv 0 \bmod d \text{ and } j \equiv i+1 \bmod m\}, 
\end{align*}
on which $\sigma_{ij}(T)$ differs from $\sigma_{ij}(T_1)$, and leaves the complement $\Lambda_d = \Lambda \setminus (\Lambda_d^+ \cup \Lambda_d^-)$ invariant.

\begin{lemma}\label{duality}
Let $T = \{P,Q,R\}$ be a Hermitian triangle.
Let $(i,j) \in \Lambda$, and write $I(i,j) = (i^*,j^*)$.
Let $D = iP+jQ + \sigma_{ij} R$ and $D^* = i^* P + j^* Q + \sigma_{i^*j^*} R$, with $\sigma_{ij}$ as in the statement of Theorem \ref{thm:sigma}. 
Then there is an isomorphism of vector spaces
\[
L(D) / L(D-R)  \longrightarrow L(K-D^*+Q+R) / L(K-D^*+Q) .
\]
\end{lemma}

\begin{proof}
With $\sigma_{ij}$ as in Theorem \ref{thm:sigma}, it follows that 
\begin{equation}\label{dual_sigmas}
  \sigma_{ij}(T) + \sigma_{i^*j^*}(T) 
= 2g - 1 - 3m .
\end{equation}
for all triangles $T$ and all $(i,j) \in \Lambda$.

The desired isomorphism is induced by multiplication by the function $y_{PQ}/y_{PR}^3$, which has divisor $2mP + mQ - 3mR$.
Indeed, for $f \in L(D)$, we find, with the help of \eqref{dual_sigmas}, that
\begin{align*}
\operatorname{div}(f \cdot y_{PQ}/y_{PR}^3) 
&\geq -D + 2mP + mQ - 3mR \\
&= (2m-i)P + (m-j)Q - (\sigma_{ij} + 3m)R \\
&= i^* P + (j^*-1)Q + (2g-1 - \sigma_{i^*j^*}) R \\
&= -(K - D^* + Q + R) ,
\end{align*}
so $f\cdot y_{PQ}/y_{PR}^3 \in L(K-D^*+Q+R)$.
The same argument shows that $L(D-R)$ consists of those functions mapped to $L(K-D^*+Q)$.
\end{proof}

\begin{proof}[Proof of Theorem \ref{thm:sigma}]

By Lemmas \ref{fg_lemma} and \ref{duality}, it will suffice to produce, for each $(i,j) \in \Lambda$, a function $f_{ij}$ in $L(K-D+Q+R) \setminus L(K-D+Q)$, where $D = iP + jQ + \sigma_{ij}R$. Now
\[
  K-D+Q+R
= - iP - (j-1)Q - (\sigma_{ij} - (2g-1))R ,
\]
so we are looking for $f_{ij}$ to satisfy
\[
  v_P(f_{ij}) \geq i,
\qquad
  v_Q(f_{ij}) \geq j-1,
\qquad
  v_R(f_{ij}) = \sigma_{ij} - (2g-1),
\]
and to be regular elsewhere. With $\sigma_{ij}$ as in the statement of the theorem,
\[
  \sigma_{ij} - (2g-1) = - 1 - jq + \epsilon_{ij},
\]
where $\epsilon_{ij} = 0$ if $(i,j) \in \Lambda_d$ and $\epsilon_{ij} = \pm 1$ if $(i,j) \in \Lambda_d^\pm$.
As it will be more convenient to work in $\mc O_{\mc H,P} = \F_q[x,y]$, instead of $f_{ij}$ itself we look for $g_{ij} = f_{ij} y_{PR}^j$ satisfying
\begin{equation}\label{PQR}
v_P(g_{ij}) \geq (i-j) - jq, \qquad
v_Q(g_{ij}) \geq j-1, \qquad
v_R(g_{ij}) = j-1 + \epsilon_{ij} . 
\end{equation}

Let $w_d$ be as in Lemma \ref{shift}. Given $g_{ij}$ as desired, the function $g_{ij} w_d^k$ satisfies the conditions in \eqref{PQR} but with $(i,j)$ replaced with $(i+kd,j+kd)$. It follows that we only need to find $g_{ij}$ for $(i,j) \in \Lambda$ with $j \leq d$.
If $T$ is collinear, i.e., if $d=1$, then the function $g_{i1} = 1$ is as desired. This completes the proof of part (a) of the theorem.
We now assume that $d>1$, so that $R = P_{\alpha\alpha}$.

\def\dd{3}
\def\mm{6}
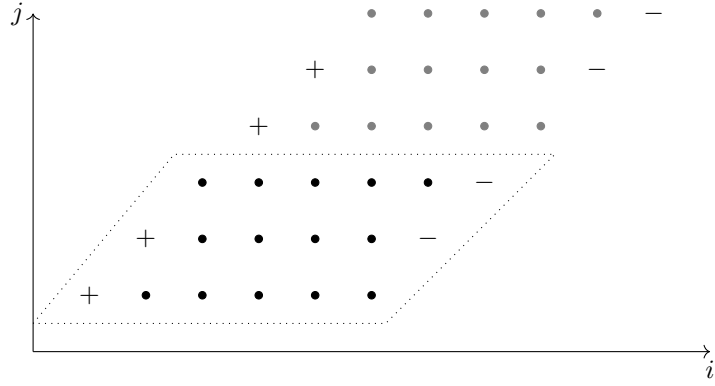
\begin{figure}[ht]
\[
\begin{tikzpicture}[scale=.75]
\draw[->] (0,0)--(2*\mm,0) node[below]{$i$};
\draw[->] (0,0)--(0,\mm) node[left]{$j$};
\draw[dotted] (0,.5)--(\dd-.5,\dd+.5)--(\mm+\dd+.25,\dd+.5)--(\mm+.25,.5)--(0,.5);

\foreach \point in {(2, 1), (3, 1), (4, 1), (5, 1), (6, 1), (3, 2), (4, 2), (5, 2), (6, 2), (3, 3), (4, 3), (5, 3), (6, 3), (7, 3)} {
	\node[draw,circle,inner sep=1pt,fill] at \point {};
}
\foreach \point in {(5, 4), (6, 4), (7, 4), (8, 4), (9, 4), (6, 5), (7, 5), (8, 5), (9, 5), (6, 6), (7, 6), (8, 6), (9, 6), (10, 6)} {
	\node[draw,circle,inner sep=1pt,fill, color=gray] at \point {};
}
\foreach \point in {(1, 1), (2, 2), (4,4), (5,5)} {
	\node[] at \point {$+$};
}
\foreach \point in {(7, 2), (8, 3), (10, 5), (11, 6)} {
	\node[] at \point {$-$};
}
\end{tikzpicture}
\]
\caption{The set $\Lambda = \Lambda_d \cup \Lambda_d^+ \cup \Lambda_d^-$ for $(q,d)=(5,3)$, with elements of $\Lambda_d^\pm$ represented by $\pm$, and elements of $\Lambda_d$ by $\bullet$. The dotted border surrounds the set of $(i,j) \in \Lambda$ with $j \leq d$.}
\label{fig_gamma}
\end{figure}

We divide the remainder of the proof into three cases corresponding to the value of $\epsilon_{ij}$, i.e., whether $(i,j)$ in $\Lambda_d$, $\Lambda_d^+$, or $\Lambda_d^-$.
See Figure \ref{fig_gamma} for a visual representation of these subsets of $\Lambda$ for $(q,d) = (5,3)$.

\textbf{Case 1.} Suppose that $(i,j) \in \Lambda_d$ and $j \leq d$.
Then either $i+1 \leq j \leq j+q-1$, or $(i,j)$ is one of the two points $(d,d)$ or $(q+1,1)$.
For the two isolated cases, one can quickly check that $g_{ij} = (x-y)^{j-1}$ satisfies \eqref{PQR}.
Now suppose that $j+1 \leq i \leq j+q-1$.
By \eqref{PQR}, it will suffice to produce $g_{ij} \in L((1+(j-1)q)P)$ satisfying $v_Q(g_{ij}) = v_R(g_{ij}) = j-1$.
Since $yy_\alpha = y^2-\alpha^qxy + \alpha^qy$, there is a polynomial $h(x)$ of degree $\leq j-2$ such that
\[
  (x-y)^{j-1} \equiv x^{j-1} + y h(x) \mod yy_\alpha .
\]
Let $g_{ij}$ be the function on the right hand side of this congruence.
Then  
\[
v_\infty(g_{ij}) \geq \min( v_\infty(x^{j-1}), v_\infty(y h(x) ))
=  - 1 - (j-1)q .
\]

Moreover, since $yy_\alpha$ vanishes to order $q+1 > j-1$ at each of $Q$ and $R$, and $(x-y)^{j-1}$ vanishes with order $j-1$ at these two points, we have $v_Q(g_{ij}) = v_R(g_{ij}) = j-1$.

\textbf{Case 2.} Suppose that $(i,j) \in \Lambda_d^+$ with $j \leq d$. Then $i=j$ and $1 \leq j \leq d-1$. We want to find $g_{ij}$ satisfying
\[
v_P(g_{ij}) \geq - jq, \qquad
v_Q(g_{ij}) \geq j-1, \qquad
v_R(g_{ij}) = j . 
\]

Since $\ord(1-\alpha) = d > j$, the polynomial $a(t) = (t-\alpha)^j - t^j$ satisfies $a(1) \neq 0$.
Furthermore, the polynomial $b(t) = a(t) - a(1) t^{j-1}$ is nonzero and satisfies $b(1) = 0$, so that there is $c(t)$ of degree $\deg b(t) - 1 \leq j-2$ such that 
\[
  (t-\alpha)^j - t^j - a(1) t^{j-1} = (t-1) c(t) .
\]

We now substitute $t = x$ and recall from \eqref{yalpha} that $x-1 = \alpha^{-q}(y-y_\alpha)$ to obtain
\begin{align*}
  (x-\alpha)^j 
&= x^j + a(1) x^{j-1} + \alpha^{-q}(y-y_\alpha) c(x)  \\
&\equiv x^j + a(1) x^{j-1} + \alpha^{-q} y c(x) \mod y_\alpha .
\end{align*}

Let $g_{ij}$ be the function on the right hand side of this congruence.
Then
\[
  v_P(g_{ij}) \geq \min( -jq, -(q+1) - q \deg(c) ) = -jq .
\]
Moreover, $v_Q(g_{ij}) = v_Q(x^{j-1}) = j-1$ since $v_Q(y c(x)) \geq v_Q(y) = q+1$, and $v_R(g_{ij}) = v_R( (x-\alpha)^j ) = j$ since $j < q +1 = v_R(y_\alpha)$.

\textbf{Case 3.} Suppose that $(i,j) \in \Lambda_d^-$ with $j \leq d$, so that $i = j+q$ and $2 \leq j \leq d$.
We want to find $g_{ij}$ satisfying
\[
v_P(g_{ij}) \geq - (j-1)q, \qquad
v_Q(g_{ij}) \geq j-1, \qquad
v_R(g_{ij}) = j-2 . 
\]

The proof for this case is similar to the previous one. 

Since $\ord(1-\alpha) = d > j-1$, the polynomial $a(t) = t^{j-1} - (t-\alpha)^{j-1}$ satisfies $a(1) \neq 0$.
Moreover, the polynomial 
\[
b(t) = a(t) - \frac{a(1)}{(1-\alpha)^{j-2}} (t-\alpha)^{j-2}
\]
is nonzero and has a root at 1, so there is $c(t)$ of degree $\deg b(t) - 1 \leq j-3$ such that

\[
  b(t) = t^{j-1} - (t-\alpha)^{j-1} - \frac{a(1)}{(1-\alpha)^{j-2}} (t-\alpha)^{j-2} = (t-1) c(t) .
\]

By taking $t=x$ and using \eqref{yalpha} to replace $x-1$ by $\alpha^{-q}(y-y_\alpha)$, it follows that 
\[
  x^{j-1}
\equiv (x-\alpha)^{j-1} + \frac{a(1)}{(1-\alpha)^{j-2}} (x-\alpha)^{j-2} + \alpha^{-q} y_\alpha c(x) \mod y.
\]

Let $g_{ij}$ be the right hand side of this congruence. Then 
\[
v_P(g_{ij}) \geq \min( -(j-1)q, -(q+1) + q \deg(c) ) = -(j-1)q .
\]
Moreover, $v_Q(g_{ij}) = j-1$ and $v_R(g_{ij}) = j-2$. \qedhere
\end{proof}

\begin{corollary}\label{corollary:S3}
Let $T = \{P,Q,R\}$ be a Hermitian triangle.
Then the dimension of a divisor $aP + bQ + cR$ is invariant under any permutation of $(a,b,c)$. 
In particular, the semigroup $H(P,Q,R) \subseteq \N^3$ is invariant under the action of the symmetric group $S_3$.
\end{corollary}

\begin{proof}
This follows from Lemma \ref{dim_lemma} and the fact that $\sigma_{ij}(T)$ is independent of the ordering of $P,Q,R$.
\end{proof}

Since the dimensions of all divisors supported on $T$ are determined by the function $(i,j) \mapsto \sigma_{ij}(T)$, and this function is determined by the type of the triangle $T$, we have the following additional corollary of Theorem \ref{thm:sigma}.

\begin{corollary}\label{type_corollary}
Two Hermitian triangles of the same type have the same Weierstrass semigroup.
\end{corollary}

\section{Counting Weierstrass gaps} \label{section:counting}
The goal of this section is to determine the cardinality of the Weierstrass gap set $G(P,Q,R) = \N^3 \setminus H(P,Q,R)$ for any Hermitian triangle $T = \{P,Q,R\}$.
For $q>3$ we will see that this distinguishes triangles of different types, thereby yielding a proof for Theorem \ref{main_theorem}.
As a byproduct of our approach, we also obtain the cardinality of the pure gap set of any triple of rational points on the Hermitian curve. 

Let $T$ be of type $d$, and let $N(T)$ denote the size of the gap set $G(P,Q,R)$.
The set $G(P,Q,R)$ consists of all nonnegative integer triple $(a,b,c)$ such that the divisor $D = aP + bQ + cR$ has a basepoint at one of the points $P,Q,R$.
To determine when this occurs, we can use the criterion given in Proposition \ref{basepoint}, which states that $D$ has a basepoint at $Q$ if and only if $\sigma_{ab}(T) > c$, with $\sigma_{ab}$ defined as in Section \ref{sec:sigma}. 

Writing $a = i+ a_1m$, $b = j+b_1m$, and $c = k+c_1m$ with $0 \leq i,j,k < m$, the inequality $\sigma_{ab}(T) > c$ may be reexpressed using \eqref{eqn:sigma_ijm} as
\[
  a_1 + b_1 + c_1 < \frac 1m (\sigma_{ij}(T) - k) .
\]
The number of nonnegative solutions $(a_1,b_1,c_1)$ of this inequality is $\binom{s(i,j,k)+3}{3}$, where
\begin{equation}\label{sd}
  s_d(i,j,k) = \left\lfloor \frac{\sigma_{ij}(T) - k- 1}{m} \right\rfloor ,
\end{equation}
and these solutions correspond to effective divisors $D = aP + bQ + cR$ with $(a,b,c) \equiv (i,j,k) \bmod m$, and which have a basepoint at $Q$.

Since cyclically permuting $(a,b,c)$ in the criterion above replaces the basepoint $Q$ by $P$ or $R$, we conclude that the divisor $D$ realizes a Weierstrass gap if and only if 
\begin{equation}\label{gap_criterion}
  a_1 + b_1 + c_1 < \frac 1m \max \left( \sigma_{ij}(T) - k, \sigma_{ki}(T) - j, \sigma_{jk}(T) - i \right) .
\end{equation}
Therefore, for each $(i,j,k)$ with $0 \leq i,j,k < m$, the nonnegative integer solutions $(a_1,b_1,c_1)$ of \eqref{gap_criterion} correspond to Weierstrass gaps $D = aP + bQ + cR$ with $(a,b,c) \equiv (i,j,k) \bmod m$.
Counting these solutions and summing over $(i,j,k)$ yields
\[
  N(T) = \sum_{i,j,k \bmod m} \max_{\tau} \binom{s_d^\tau(i,j,k)+3}{3}, 
\]
where $\tau$ runs over cyclic permutations of the three inputs of $s_d$.

\begin{proposition}\label{NT1}
Let $T_1=\{ P, Q, R\}$ be a collinear triangle. Then the cardinality of its Weierstrass gap set is
\[
 \#G(P,Q,R)= N(T_1) = \frac{1}{24} (2q^6 - 6q^5 + 9q^4 + 12q^3 - 35q^2 + 18q) .
\]
\end{proposition}

\begin{proof}

We compute $N(T_1)$ as the sum of contributions corresponding to triples $(i,j,k)$ with $0 \leq i,j,k < m$, as described above.
Because the Weierstrass semigroup is invariant under the action of $S_3$, we work under the assumption that $i \leq j \leq k$, and weight the resulting contributions by the number of distinct permutations of each triple $(i,j,k)$.

We begin by simplifying the expression
\[
  s_1(i,j,k) = \left\lfloor \frac{\sigma_{ij}(T_1) - k- 1}{m} \right\rfloor .
\]
By \eqref{eqn:sigma_ijm} and \eqref{eqn:sigma}, the values values of $\sigma_{ij}(T_1)$ for $0 \leq i,j \leq m-1$ are given by
\begin{equation}\label{eqn:sigma_piecewise}
\sigma_{ij}(T_1)
= \begin{cases}
0,  & j = 0,  \\
2g-2-jq+m\delta_{j > i},  & 0 < j \leq q ,
\end{cases}
\end{equation}
where $\delta_{j > i} =1$ if $j > i$ and is $\delta_{j > i}=0$ otherwise.
If $j = 0$, then $\sigma_{ij}(T_1) = 0$, so that $s_1(i,j,k) = -1$. 
A negative value here means that $(i,j,k)$ makes no contribution to $N(T_1)$.
If $j > 0$, then by \eqref{eqn:sigma_piecewise} and the fact that $-m < j-k-1 \leq m-2$, we have
\[
s_1(i,j,k) =
q-2 - j + \delta_{j>i} + \left\lfloor \frac{j-k-1}{m}\right\rfloor 
= q-3-j+ \delta_{j>i} + \delta_{j > k} .
\]
We break the remainder of the proof into four cases, depending on whether or not the the inequalities $i \leq j$ and $j \leq k$ are strict.

First assume that $i=j=k$. Then either $j=0$ and $s_1(j,j,j) = -1$, or $j>0$ and $s_1(j,j,j) = q - 3 - j$. The total contribution of triples of this form is therefore
\[
\sum_{j=0}^q \binom{s_1(j,j,j)+3}{3}
= \sum_{j=1}^q \binom{q-j}{3}
= \binom{q}{4}.
\]

Now assume that $i=j<k$. Then
\[
s_1(j,j,k) = s_1(k,j,j) 
= \begin{cases}
-1, & j = 0,\\
q-3-j, & j > 0,
\end{cases}
\]
while in in either case, $s_1(j,k,j) = q-1-k$.
It follows that
\[
\max_\tau s_1^\tau(j,j,k) = 
\begin{cases}
q-1-k, &  j = 0 \text{ or } k = j+1, \\
q-3-j, & j > 0 \text{ and } k > j+1 .
\end{cases}
\]
These triples have three distinct permutations, so their total weighted contribution is 3 times
\begin{multline*}
  \sum_{k=1}^q \binom{q+2-k}{3} 
+ \sum_{k=2}^q \binom{q+2-k}{3} 
+ \sum_{j=1}^{q-2} (q-j-1) \binom{q-j}{3} \\
= \frac{1}{30} (q^5 - 5 q^4 + 20 q^3 - 25 q^2 + 9 q) .
\end{multline*}

Now assume that $i < j = k$.
Then $s_1(j,i,j) = -1$ if $i = 0$ and otherwise
\[
s_1(j,i,j) = q-3-i \geq q-2-j = s_1(i,j,j) = s_1(j,j,i) .
\]
It follows that
\[
\max_\tau s_1^\tau(i,j,j) = 
\begin{cases}
\max(-1, q-2-j), & i = 0 , \\
q-3-i, & i > 0 . 
\end{cases}
\]
These triples have three distinct permutations, so their total weighted contribution is 3 times
\[
  \sum_{j=1}^{q} \binom{q+1-j}{3}
+ \sum_{i=1}^{q-1} (q-i) \binom{q-i}{3}
= \frac{1}{30}(q^5 - 5q^4 + 15q^3 - 10q^2 - q) .
\]

Finally, assume that $i < j < k$.
If $i = 0$, then $s_1(k,i,j) = -1$, and 
\[
  \max_\tau s_1^\tau(i,j,k)
= \max(q-1-k, q-2-j, -1)
= q-2-j .
\]
On the other hand, if $i > 0$, then 
\[
  \max_\tau s_1^\tau(i,j,k)
= \max(q-1-k, q-2-j, q-3-i)
= q-3-i .
\]
These triples have six distinct permutations, so their total weighted contribution is 6 times
\begin{multline*}
\sum_{j=1}^{q-1} (q-j) \binom{q+1-j}{3}
+ \sum_{i=1}^{q-2} \binom{q-i}{2} \binom{q-i}{3} \\
= \frac{1}{360} (5q^6 - 27q^5 + 80q^4 - 135q^3 + 95q^2 - 18q) .
\end{multline*}

Adding up these weighted contributions gives the desired formula for $N(T_1)$. \qedhere
\end{proof}

Before we count the Weierstrass gaps for a triangle $T_d$ of type $d \neq 1$, we examine the situations in which the quantities $s_1(i,j,k)$ and $s_d(i,j,k)$ may differ for a given triple $(i,j,k)$.

\begin{lemma}\label{s1sd}
Let $0 \leq i,j,k < m$ and let $d$ be a divisor of $q+1$. Then $s_1(i,j,k) \neq s_d(i,j,k)$ only in one of the following three cases:
\begin{enumerate}
\item $(i,j,k) = (q,0,q)$, in which case $s_1(q,0,q) = -1$ and $s_d(q,0,q) = -2$;
\item $(i,j,k) = (i, i+1, i)$ with $i \not\equiv 0 \bmod d$, in which case $s_d(i,j,k) = s_1(i,j,k) - 1$;
\item $(i,j,k) = (i,i,i)$ with $i \not\equiv 0 \bmod d$, in which case $s_d(i,j,k) = s_1(i,j,k) + 1$.
\end{enumerate}
\end{lemma}

\begin{proof}

Recall from Theorem \ref{thm:sigma} that $\sigma_{ij}(T_d) = \sigma_{ij}(T_1) + \epsilon_{ij}$, where
\[
\epsilon_{ij} = 
\begin{cases}
    \phantom{-}1 & i\not\equiv 0 \bmod d \text{ and } j\equiv i \bmod m, \\
    -1 & i\not\equiv 0 \bmod d \text{ and } j\equiv i+1 \bmod m, \\
    \phantom{-}0 & \text{otherwise.}
\end{cases}
\]

If $j = 0$, then $\sigma_{ij}(T_1) = 0$, and 
\[
  \sigma_{ij}(T_d) 
= \epsilon_{ij}
= \begin{cases}
-1, & i = q \text{ and } d > 1, \\
\phantom{-}0, & \text{otherwise}.
\end{cases} 
\]
It follows that $s_1(i,0,k)$ and $s_d(i,0,k)$ are both negative.
In particular, $s_1(i,0,k) = -1$ and 
\[
s_d(i,0,k) = 
\begin{cases}
-2, & (i,j,k) = (q,0,q), \\
-1, & \text{otherwise}.
\end{cases}
\]

Assume now that $j \neq 0$, so that $\sigma_{ij}(T_d) = 2g-2 - jq + m\delta_{j > i} + \epsilon_{ij}$.
Then
\[
  s_d(i,j,k) 
= q-2 - j + \delta_{j > i} + \left\lfloor \frac{j - k - 1 + \epsilon_{ij}}{m}  \right\rfloor .
\]
Therefore, $s_d(i,j,k) \neq s_1(i,j,k)$ only if
\[
\left\lfloor \frac{j - k - 1}{m}  \right\rfloor
\neq \left\lfloor \frac{j - k - 1 + \epsilon_{ij}}{m}  \right\rfloor .
\]

Because $0\leq i,j,k < m$ and $j \neq 0$, there are only two scenarios in which this last condition occurs, namely when $j = k+1$ and $\epsilon_{ij} = -1$, and when $j = k$ and $\epsilon_{ij} = +1$.
In the first case, we have $(i,j,k) = (i,i+1,i)$ with $i \not\equiv 0 \bmod d$, and $s_d(i,j,k) = s_1(i,j,k) - 1$.
In the second case, we have $(i,j,k) = (i,i,i)$, with $i \not\equiv 0 \bmod d$, and $s_d(i,j,k) = s_1(i,j,k) + 1$. \qedhere
\end{proof}

\begin{lemma}\label{maxperm}
Let $d$ be a divisor of $q+1$, and let $0\leq i,j,k < m$.
Then there is a cyclic permutation of $(i,j,k)$ which simultaneously maximizes the quantities $s_1(i,j,k)$ and $s_d(i,j,k)$.
\end{lemma}

\begin{proof}

We assume, without loss of generality, that $(i,j,k)$ satisfies $s_1(i,j,k) \geq s_1(i',j',k')$ for any any cyclic permutation $(i',j',k')$ of $(i,j,k)$, and furthermore that $s_d(i,j,k) \geq s_d(i',j',k')$ whenever $s_1(i,j,k) = s_1(i',j',k')$.
Unless $(i,j,k) = (0,0,0)$, this implies that $j \neq 0$.

Suppose that $(i',j',k')$ is a cyclic permutation of $(i,j,k)$ satisfying $s_d(i',j',k') > s_d(i,j,k)$.
Then according to the assumption on $(i,j,k)$, we must have $s_1(i',j',k') < s_1(i,j,k)$.
Since $s_1(i,j,k)$ and $s_d(i,j,k)$ may differ by at most 1, it follows that $s_d(i,j,k) = s_1(i,j,k) - 1$ and $s_d(i',j',k') = s_1(i',j',k') + 1$. However,  by Lemma \ref{s1sd}, this implies that $(i,j,k) = (i,i+1,i)$ and $(i',j',k') = (i',i',i')$, a contradiction. \qedhere

\end{proof}

\begin{proposition}\label{NTd}
Let $T_d=\{ P, Q, R \}$ be a triangle of type $d$. Then the cardinality of its Weierstrass gap set is
\[
 \#G(P,Q,R)= N(T_d) = N(T_1) - \frac{d-1}{6d} (q+1)(2q^2 - 5q + 5-d).
\]
\end{proposition}

\begin{proof}

For any divisor $t$ of $q+1$, let 
\[
  c_{t}(i,j,k) = \max_{\tau} \binom{s_{t}^\tau(i,j,k)+3}{3} .
\]
To obtain $N(T_d) - N(T_1)$, it will suffice to sum the differences $c_d(i,j,k) - c_1(i,j,k)$.
We will sum these differences over orbits of triples $(i,j,k)$ under cyclic permutation, weighted by orbit size.

Let $(i,j,k)$ satisfy $0 \leq i,j,k < m$.
By Lemma \ref{maxperm}, we assume that $(i,j,k)$ simultaneously maximizes $s_1(i,j,k)$ and $s_d(i,j,k)$ with respect to cyclic permutations, so that 
\[
  c_{t}(i,j,k) = \binom{s_{t}(i,j,k)+3}{3} , 
\qquad t = 1,d.
\]
Then by Lemma \ref{s1sd}, we have $c_d(i,j,k) - c_1(i,j,k) = 0$ unless $(i,j,k) = (i,i+1,i)$ or $(i,j,k) = (i,i,i)$ with $i \not\equiv 0 \bmod d$.
In the first case, we have
\[
  c_d(i,i+1,i) - c_1(i,i+1,i) = \binom{q-i}{3} - \binom{q+1-i}{3} = - \binom{q-i}{2} .
\]
For each $1\leq i \leq q-1$ with $i \not\equiv 0 \bmod d$, there three triples in each of these orbits which yield this same contribution.

In the second case, we have
\[
  c_d(i,i,i) - c_1(i,i,i) = \binom{q+1-i}{3} - \binom{q-i}{3} = \binom{q-i}{2} .
\]
There is a single contribution of this type for each $1\leq i \leq q$ with $i \not\equiv 0 \bmod d$.

It follows that 
\[
  N(T_d) - N(T_1)
= \sum_{\substack{0 \leq i \leq q \\ i\not\equiv 0\bmod d} } \left\{ \binom{q-i}{2} - 3 \binom{q-i}{2} \right\} 
= -2 \sum_{\substack{0 \leq i \leq q \\ i\not\equiv 0\bmod d} } \binom{q-i}{2} .
\]
To complete the proof, we calculate
\begin{align*}
\sum_{\substack{0 \leq i \leq q \\ i\not\equiv 0\bmod d} } \binom{q-i}{2}
&= \sum_{i=0}^q {{q-i}\choose{2}} - \sum_{j=0}^{(q+1)/d-1}{{q-jd}\choose{2}} \\
&= \frac 16(q^3-q) - \frac{1}{12d} (q+1) (2q^2-5q+5+3dq-6d+d^2)\\
&= \frac{d-1}{12d} (q+1)(2q^2 - 5q + 5-d) . \qedhere
\end{align*}

\end{proof}

We now restate and prove the main theorem.

\begin{prevtheorem}
For $q > 3$, the number of distinct Weierstrass semigroups of triples of rational points on the Hermitian curve over $\F_{q^2}$ is equal to the number of positive divisors of $q+1$.
For $q=2$, there is a single three-point semigroup, and for $q=3$, there are two.
\end{prevtheorem}

\begin{proof}

Let $d$ and $e$ be distinct divisors of $q+1$.
Then the Weierstrass gap sets of triangles of types $d$ and $e$ have the same size if and only if $h_d(q) = h_e(q)$, where
\[
h_d(q) = \frac{d-1}{12d} (q+1)(2q^2 - 5q + 5-d) .
\]
Elementary simplification reveals that this is equivalent to saying that $de = 2q^2 - 5q + 5$. 
But since $d$ and $e$ are distinct divisors of $q+1$, we have
\[
de \leq \frac 12 (q+1) \cdot (q+1) < 2q^2  - 5q + 5 ,
\]
provided that $q > 3$. 
Therefore, for $q>3$, triangles of distinct types have different Weierstrass semigroups.
The first part of theorem now follows from Corollary \ref{type_corollary}. 

For $q=2$, we have $N(T_1) = N(T_3) = 3$.
It follows that 
\[
H(P,Q,R) = \N^3 \setminus \{ (1,0,0), (0,1,0), (0,0,1) \}
\]
for any $P,Q,R \in \mc H(\F_4)$.

For $q=3$, we have $N(T_1) = 33$ and $N(T_2) = N(T_4) = 31$.
Moreover, one may check directly that the gap sets coincide for triangles of types 2 and 4. \qedhere

\end{proof}

For any Hermitian triangle $T = \{P,Q,R\}$, the proof of Proposition \ref{NTd} may be adapted to provide formulas for the number $N_r(T)$ of effective divisors $D = aP + bQ + cR$ which have basepoints at $r$ specified points of $T$, for $r = 1,2,3$. In particular, $N_3(T) = \#G_0(P,Q,R)$ is the number of pure gaps for $(P,Q,R)$, i.e., the number of divisors $D = aP+bQ+cR$ with basepoints at all three of the points $P$, $Q$, and $R$.

\begin{proposition}\label{Nr}
Let $T = \{P,Q,R\}$ be a Hermitian triangle of type $d$, and let 
\[
  h_d(q) = \frac{d-1}{12d} (q+1)(2q^2 - 5q + 5-d) .
\]
\begin{enumerate}[(a)]
\item The number of effective divisors of the form $aP + bQ + cR$ with one specified basepoint in $T$ is
\[
   N_1(T) = \frac{1}{24}(q^6 - 2q^5 + 4q^3 - q^2 - 2q) .
\]
\item The number of effective divisors of the form $aP + bQ + cR$ with two specified basepoints in $T$ is
\[
   N_ 2(T) = \frac{1}{120} (2q^6 + q^5 - 20q^4-25q^3+138q^2-96q)  + h_d(q).
\]
\item The number $N_3(T) = \#G_0(P,Q,R)$ of pure gaps for $(P,Q,R)$ is 
\[
  \#G_0(P,Q,R) = \frac{1}{120} (q^6+3q^5-15q^4-75q^3+254q^2-168q) + h_d(q) .
\]
\end{enumerate}
\end{proposition}

\begin{proof}

As the proof is similar to that of Propositions \ref{NT1} and \ref{NTd}, we give only a sketch.
The setup for counting $N_r(T)$ is like that for $N(T)$, except that \eqref{gap_criterion} is replaced by
\[
  a_1 + b_1 + c_1 < \frac 1m \min_\tau ( \sigma_{\tau(i)\tau(j)}(T) - \tau(k)) ,
\]
where $\tau$ runs over $r$ distinct cyclic permutations of $(i,j,k)$ corresponding to the $r$ specified points of $T$.
In particular, we have 
\[
  N_r(T) = \sum_{i,j,k \bmod m} \min_{\tau} \binom{s_d^\tau(i,j,k)+3}{3} ,
\]
where $s_d(i,j,k)$ as defined as in \eqref{sd}.
These minimums may be determined on a case-wise basis.
\qedhere

\end{proof}

Note that these values of $N_r(T)$ in Proposition \ref{Nr} satisfy the inclusion-exclusion relationship
\[
  N(T) = 3N_1(T) - 3N_2(T) + N_3(T),
\]
with $N(T)$ as in Propositions \ref{NT1} and \ref{NTd}, so they could be used to provide an alternate, albeit longer, proof of the formula for $N(T)$.

\section{Minimal generating sets} \label{section:semigps}

In \cite{Matthews04}, the first author introduced the idea of minimal generating sets as a way of describing the Weierstrass semigroup of an $n$-tuple of points on a curve over a field $\F$.  
These sets generate the Weierstrass semigroup with respect to the operation of taking coordinate-wise maximimums.
We recount the relevant definitions and results here before using them to describe the semigroups of triples of rational points on the Hermitian curve. 

Consider the partial order on $\mathbb{N}^n$ given by $(u_1,\hdots,u_n) \leq (v_1,\hdots, v_n)$ if $u_i\leq v_i$ for all $i$.

\begin{definition*}
Let $u^{(1)},\hdots, u^{(t)}\in \mathbb{N}^n$, and write $u^{(k)}=(u^{(k)}_1, \ldots, u^{(k)}_n)$. The \textit{least upper bound} of $u^{(1)}, \ldots, u^{(t)}$ is given by 
\[
\lub(u^{(1)}, \ldots, u^{(t)}) = (\max\{ u^{(1)}_1,\ldots,u^{(t)}_1\},\hdots, \max\{u^{(1)}_n,\ldots,u^{(t)}_n\}) .
\]
\end{definition*}

Let $P_1,\hdots, P_n$ be distinct rational points on an absolutely irreducible, smooth, projective algebraic curve $\mathcal{X}$ over a finite field $\F$ with $n<|\F|$. Then the semigroup $H(P_1, \dots, P_n)$ is closed under addition as well as under the operation $\lub$ \cite{CarvalhoTorres}. Indeed, the requirement that $|\F|>n$ ensures that given functions $f_1, \dots, f_t \in \F(\mathcal{X})$ with pole divisors $(f_k)_{\infty}=u^{(k)}$, there is a linear combination $f=\sum_{k=1}^t a_k f_k$ with $a_k \in \F$ and $(f)_{\infty}=\lub(u^{(1)},\hdots, u^{(t)} )$. 

We now define a subset $\Gamma^+(P_{i_1},\ldots,P_{i_\ell})$ of $H(P_1,\ldots,P_n)$ for each $\{i_1,\ldots,i_\ell\} \subseteq \{1,\ldots,n\}$ as follows.
Let $\Z_{>0}$ denote the set of positive integers.
For $\ell = 1$ and $i \in \{1,\ldots,n\}$, let $\Gamma^+(P_i) = H(P_i)$, and for $\ell \geq 2$, let
\[
\Gamma^+(P_{i_1}, \ldots, P_{i_\ell}) =
\left\{
u \in \Z_{>0}^\ell :
\begin{array}{c}
 u \text{ is a minimal element of }\\
 \{v \in H(P_{i_1}, \ldots, P_{i_\ell}) : v_i = u_i\}\\
 \text{ for some }i \in \{1,\ldots,\ell\} \end{array}
\right\} .
\]
We note that although the one-point semigroups $\Gamma^+(P_i)$ are infinite, each set $\Gamma^+(P_{i_1}, \ldots, P_{i_\ell})$ with $\ell \geq 2$ is a subset of the product $G(P_{i_1}) \times \cdots \times G(P_{i_\ell})$ of gap sets, hence is finite.

\begin{example*}
For any two rational points $P, Q$ on the Hermitian curve, we have
\begin{align*}
\Gamma^+(P) &= \{ iq + j(q+1) : i,j \in \N \} ,  \\
\Gamma^+(P,Q) &= \{ (-i+jm, iq-jm) : 0 \leq i \leq q,  i < jm < iq\} .
\end{align*}
\end{example*}

For each $I \subseteq \{1,\hdots,n\}$ let $\iota_I$ denote the natural inclusion $\mathbb{N}^{\ell} \to \mathbb{N}^n$ into the coordinates indexed by $I$.

\begin{definition*}
The \textit{minimal generating set} of $H(P_1,\hdots, P_n)$ is 
\[
\Gamma(P_1,\hdots, P_n) 
= \bigcup_{\ell=1}^n \bigcup_{\substack{I=\{i_1,\hdots,i_\ell\} \\ i_1 < \cdots < i_\ell}} \iota_I(\Gamma^+(P_{i_1},\hdots,P_{i_\ell})).
\]
\end{definition*}

The Weierstrass semigroup $H(P_1,\hdots, P_n)$ is completely determined by the minimal generating set $\Gamma(P_1,\hdots,P_n)$ on those same points, as described below. 

\begin{theorem}\cite[Theorem 7]{Matthews04}
If $1\leq n < |\F|$, then
\[
H(P_1,\hdots,P_n) = \{\lub( u^{(1)},\ldots,  u^{(n)}):  u^{(1)},\hdots, u^{(n)} \in\Gamma(P_1,\hdots,P_n)\}.
\]
\end{theorem}

To characterize the minimal generating set for Weierstrass semigroups of Hermitian triangles, we use the following lemma of Castellanos and Tizziotti. 

\begin{lemma}\cite[Lemma 2.6]{Cas_Tiz}\label{mgs_lemma}
Let $\alpha=(\alpha_1,\hdots,\alpha_n)\in \Z_{>0}^n$. Then $\alpha \in\Gamma^+(P_1,\hdots,P_n)$ if and only if the divisor $\alpha_1P_1+\cdots+\alpha_n P_n$ is a discrepancy with respect to $P$ and $Q$ for all pairs of points $P,Q\in \{P_1,\hdots,P_n\}$.
\end{lemma}

We note that the condition that $\alpha$ have all positive entries is not explicitly stated in the original version of the lemma, but it is necessary. Indeed, it is possible for $\alpha$ to have a zero entry yet correspond to a divisor which is a discrepancy for any pair of points in $\{P_1,\ldots,P_n\}$.

For $P,Q$ two rational points on $\mc H$, it is shown in \cite[Section 4]{Duursma_Park} that the set of discrepancies for $P$ and $Q$ of the form $aP+bQ$ is
\[
  \Delta_0(P,Q) = \{ i(qQ - P) + j(mP-mQ) : 0 \leq i \leq q, j \in \Z \}.
\]
It follows that the description of $\Gamma^+(P,Q)$ in the example above, which was originally derived in \cite[Theorem 3.7]{pairs} by other means, may be obtained by an application of Lemma \ref{mgs_lemma}.

\begin{proposition}\label{min_gen_set}
Let $T=\{P,Q,R\}$ be a Hermitian triangle of type $d$.
Let $(a,b,c)\in \Z_{>0}^3$, and write $a=i+a_1m$, $b=j+b_1m$,  and $c=k+c_1m$ with $0\leq i,j,k < m$. 
Then $(a,b,c)\in \Gamma^+(P,Q,R)$ if and only if $a_1+b_1+c_1=q-2-\ell$ for some $\ell \in \{0,\ldots,q-2\}$, and 
\begin{itemize}
\item $(i,j,k) = (\ell,\ell,\ell)$ if $\ell \equiv 0 \bmod d$, or 
\item $(i,j,k)$ is a permutation of $(\ell+1,\ell,\ell)$ if $\ell \not\equiv 0 \bmod d$.
\end{itemize} 
\end{proposition}

\begin{proof}
Lemma \ref{mgs_lemma} implies that $(a,b,c)\in \Gamma^+(P,Q,R)$ exactly when $aP+bQ+cR$ is a discrepancy for all pairs of points of $T=\{P,Q,R\}$. In the notation introduced in Section \ref{sec:sigma}, this means that $(a,b,c)\in\Gamma^+(P,Q,R)$ if and only if 
\[
\begin{array}{cccc}
\sigma_{ab}(T)=c, & \sigma_{bc}(T)=a, & \text{and} & \sigma_{ca}(T)=b.
\end{array}
\]
By \eqref{eqn:sigma_ijm}, this is equivalent to 
\begin{equation}\label{eqn:mgs_criterion}
\sigma_{ij}(T)-k = \sigma_{jk}(T)-i = \sigma_{ki}(T)-j = m(a_1+b_1+c_1).
\end{equation}

By Theorem \ref{thm:sigma} and \eqref{eqn:sigma_piecewise}, we have
\begin{equation}\label{jkeps}
\sigma_{ij}(T)-k  \equiv j-k + \epsilon_{ij}(T) \mod m
\end{equation}
for all $0 \leq i,j,k < m$. 
Using this, one may check directly that \eqref{eqn:mgs_criterion} is satisfied when $(a,b,c)$ is of one of the two forms described in the statement of the proposition.

Now assume that $(i,j,k)$ satisfies \eqref{eqn:mgs_criterion}.
Then the congruence \eqref{jkeps} holds for any cyclic permutation of $(i,j,k)$.
If $i=j=k$, then $\sigma_{ii}(T)-i \equiv \epsilon_{ii} \bmod m$, so that $i \equiv 0 \bmod d$. Then setting $\ell = i$, we are in the first case in the statement of the proposition.

Now assume that $i,j,k$ are not all equal.
Then by permuting $(i,j,k)$ if necessary, we may assume that $j < i$.
Then in particular, $j$ is congruent to neither $i$ nor $i+1$ modulo $m$, so that $\epsilon_{ij} = 0$. But then \eqref{jkeps} implies that $j = k$.
Now consider the fact that
\[
  \sigma_{jk}(T) - i \equiv i - j + \epsilon_{ji} \equiv 0 \mod m .
\]
Since $i \neq j$, it follows that $\epsilon_{ji} \not\equiv 0 \bmod m$ and $\epsilon_{ji} \neq +1$.
Thus, $\epsilon_{ji} = -1$, and so $i \equiv j+1 \bmod m$ with $j \not\equiv 0 \bmod d$.
Setting $\ell = j$, we are in the second case in the statement of the proposition.

In either case, $\sigma_{ij}(T) - k = m(a_1+b_1+c_1)$ yields $a_1 + b_1 + c_1 = q-2-\ell$. \qedhere

\end{proof}

\bibliography{hermitian_semigroups}{}

\begin{thebibliography}{10}

\bibitem{ACGH}
E.~Arbarello, M.~Cornalba, P.~A. Griffiths, and J.~Harris.
\newblock {\em Geometry of algebraic curves. {V}ol. {I}}, volume 267 of {\em
  Grundlehren der Mathematischen Wissenschaften [Fundamental Principles of
  Mathematical Sciences]}.
\newblock Springer-Verlag, New York, 1985.

\bibitem{BK}
E.~Ballico and S.~J. Kim.
\newblock Weierstrass multiple loci of {$n$}-pointed algebraic curves.
\newblock {\em J. Algebra}, 199(2):455--471, 1998.

\bibitem{BR}
Edoardo Ballico and Alberto Ravagnani.
\newblock A zero-dimensional approach to {H}ermitian codes.
\newblock {\em J. Pure Appl. Algebra}, 219(4):1031--1044, 2015.

\bibitem{BQZ}
D.~Bartoli, L.~Quoos, and G.~Zini.
\newblock Algebraic geometric codes on many points from {K}ummer extensions.
\newblock {\em Finite Fields Appl.}, 52:319--335, 2018.

\bibitem{BMZ}
Daniele Bartoli, Maria Montanucci, and Giovanni Zini.
\newblock A{G} codes and {AG} quantum codes from the {GGS} curve.
\newblock {\em Des. Codes Cryptogr.}, 86(10):2315--2344, 2018.

\bibitem{CF1}
A.~Campillo and J.~I. Farr\'{a}n.
\newblock Computing {W}eierstrass semigroups and the {F}eng-{R}ao distance from
  singular plane models.
\newblock {\em Finite Fields Appl.}, 6(1):71--92, 2000.

\bibitem{CF}
A.~Campillo and J.~I. Farr\'{a}n.
\newblock Symbolic {H}amburger-{N}oether expressions of plane curves and
  applications to {AG} codes.
\newblock {\em Math. Comp.}, 71(240):1759--1780, 2002.

\bibitem{Carvalho_V}
C\'{\i}cero Carvalho.
\newblock On {$\scr V$}-{W}eierstrass sets and gaps.
\newblock {\em J. Algebra}, 312(2):956--962, 2007.

\bibitem{CarvalhoTorres}
C\'{\i}cero Carvalho and Fernando Torres.
\newblock On {G}oppa codes and {W}eierstrass gaps at several points.
\newblock {\em Des. Codes Cryptogr.}, 35(2):211--225, 2005.

\bibitem{CB}
A.~S. Castellanos and M.~Bras-Amor{\'o}s.
\newblock {W}eierstrass semigroup at $m+1$ rational points in maximal curves
  which cannot be covered by the {H}ermitian curve.
\newblock {\em Designs, Codes and Cryptography}, 2020.

\bibitem{Cas_Tiz}
A.~S. Castellanos and G.~Tizziotti.
\newblock On {W}eierstrass semigroup at {$m$} points on curves of the form
  {$f(y)=g(x)$}.
\newblock {\em J. Pure Appl. Algebra}, 222(7):1803--1809, 2018.

\bibitem{Duursma_Park}
Iwan~M. Duursma and Seungkook Park.
\newblock Delta sets for divisors supported in two points.
\newblock {\em Finite Fields Appl.}, 18(5):865--885, 2012.

\bibitem{GKL}
Arnaldo Garc\'{\i}a, Seon~Jeong Kim, and Robert~F. Lax.
\newblock Consecutive {W}eierstrass gaps and minimum distance of {G}oppa codes.
\newblock {\em J. Pure Appl. Algebra}, 84(2):199--207, 1993.

\bibitem{Garcia_Viana}
Arnaldo Garc\'{\i}a and Paulo Viana.
\newblock Weierstrass points on certain nonclassical curves.
\newblock {\em Arch. Math. (Basel)}, 46(4):315--322, 1986.

\bibitem{Geil}
Olav Geil and Ryutaroh Matsumoto.
\newblock Bounding the number of {$\Bbb F_q$}-rational places in algebraic
  function fields using {W}eierstrass semigroups.
\newblock {\em J. Pure Appl. Algebra}, 213(6):1152--1156, 2009.

\bibitem{Ishii}
Naonori Ishii.
\newblock A certain graph obtained from a set of several points on a {R}iemann
  surface.
\newblock {\em Tsukuba J. Math.}, 23(1):55--89, 1999.

\bibitem{KP}
Christoph Kirfel and Ruud Pellikaan.
\newblock The minimum distance of codes in an array coming from telescopic
  semigroups.
\newblock volume~41, pages 1720--1732. 1995.
\newblock Special issue on algebraic geometry codes.

\bibitem{Korchmaros}
G.~Korchm\'{a}ros and G.~P. Nagy.
\newblock Hermitian codes from higher degree places.
\newblock {\em J. Pure Appl. Algebra}, 217(12):2371--2381, 2013.

\bibitem{Little}
John Little.
\newblock Personal communication, 26 Jan 2002.

\bibitem{pairs}
Gretchen~L. Matthews.
\newblock Weierstrass pairs and minimum distance of {G}oppa codes.
\newblock {\em Des. Codes Cryptogr.}, 22(2):107--121, 2001.

\bibitem{Matthews04}
Gretchen~L. Matthews.
\newblock The {W}eierstrass semigroup of an {$m$}-tuple of collinear points on
  a {H}ermitian curve.
\newblock In {\em Finite fields and applications}, volume 2948 of {\em Lecture
  Notes in Comput. Sci.}, pages 12--24. Springer, Berlin, 2004.

\bibitem{PT}
Ruud Pellikaan and Fernando Torres.
\newblock On {W}eierstrass semigroups and the redundancy of improved geometric
  {G}oppa codes.
\newblock {\em IEEE Trans. Inform. Theory}, 45(7):2512--2519, 1999.

\bibitem{RS}
Hans-Georg R\"{u}ck and Henning Stichtenoth.
\newblock A characterization of {H}ermitian function fields over finite fields.
\newblock {\em J. Reine Angew. Math.}, 457:185--188, 1994.

\bibitem{St}
Henning Stichtenoth.
\newblock \"{U}ber die {A}utomorphismengruppe eines algebraischen
  {F}unktionenk\"{o}rpers von {P}rimzahlcharakteristik. {I}. {E}ine
  {A}bsch\"{a}tzung der {O}rdnung der {A}utomorphismengruppe.
\newblock {\em Arch. Math. (Basel)}, 24:527--544, 1973.

\bibitem{TT}
Wanderson Ten\'{o}rio and Guilherme Tizziotti.
\newblock Generalized {W}eierstrass semigroups and {R}iemann-{R}och spaces for
  certain curves with separated variables.
\newblock {\em Finite Fields Appl.}, 57:230--248, 2019.

\bibitem{TizziottiCastellanos}
G.~Tizziotti and A.~S. Castellanos.
\newblock Weierstrass semigroup and pure gaps at several points on the {$GK$}
  curve.
\newblock {\em Bull. Braz. Math. Soc. (N.S.)}, 49(2):419--429, 2018.

\end{thebibliography}
\bibliographystyle{plain}

\end{document}